\documentclass[noams]{alggeom}
\usepackage{amssymb}
\usepackage{amsmath}
\usepackage{graphicx}
\usepackage{array,hhline}
\usepackage{enumitem}
\usepackage{mathtools}

\newtheorem{lemma}{Lemma}[section]
\newtheorem{theorem}[lemma]{Theorem}
\newtheorem{proposition}[lemma]{Proposition}

\newtheorem{corollary}[lemma]{Corollary}
\theoremstyle{definition}
\newtheorem{definition}[lemma]{Definition}

\numberwithin{equation}{section}
\numberwithin{figure}{section}

\renewcommand{\theenumi}{\roman{enumi}}

\newcommand{\Xset}{\mathcal{X}}

\DeclarePairedDelimiter\floor{\lfloor}{\rfloor}

\renewcommand{\contentsname}{Contents\\{\footnotesize\normalfont}}
\begin{document}

\title{On a scale of criteria on $n$-dependence}
\author{Davit Voskanyan}
\email{ysudav@gmail.com}
\address{Department of Informatics and Applied Mathematics, Yerevan State University, A. Manukyan St. 1, 0025 Yerevan, Armenia}

\begin{abstract}
In this paper we prove that a planar set $\mathcal{X}$ of at most $mn-1$ points, where $m \le n$, is $\kappa$-dependent, if and only if there exists a number r, $1 \le r \le m-1$, and an essentially $\kappa$-dependent subset $\mathcal{Y} \subset \mathcal{X}$, $\#\mathcal{Y} \ge rs$, where $r + s - 3 = \kappa$, belonging to an algebraic curve of degree $r$, and not belonging to any curve of degree less than $r$. 
Moreover, if $\#\mathcal{Y} = rs$ then the set $\mathcal{Y}$ coincides with the set of intersection points of some two curves of degrees $r$ and $s$, respectively. 

Let us mention that the first three criteria of the scale, for $m=1,2,3,$ are well-known results.
\end{abstract}

\maketitle

% \vspace*{6pt}\tableofcontents  % for this guide only.
% A table of contents should normally not be included

\section{Introduction\label{Sec1}, $n$-independence\label{Sec2}}

Denote by $\Pi_n$ the space of bivariate algebraic polynomials of total degree less than or equal to $n$. Its dimension is given by
\begin{equation*}\label{eq:dim}
  N := \dim \Pi_n = \binom{n+2}{2}.
\end{equation*}
A plane algebraic curve is the zero set of some bivariate
polynomial. To simplify notation, we shall use the same letter $p$,
say, to denote the polynomial $p$ and the curve given by the
equation $p(x,y)=0$. More precisely, suppose $p$ is a polynomial
without multiple factors. Then the plane curve defined by the
equation $p(x,y)=0$ shall also be denoted by $p$. So lines, conics, and cubics are equivalent to polynomials of degree
$1$, $2$, and $3$, respectively.
% a reducible conic is a pair of
% lines, and a reducible cubic is a triple of lines, or consists of a
% line and an irreducible conic.

Suppose a set of $k$ distinct points is given:
\begin{equation*}\label{eq:knset}{\mathcal X}_k =
\{(x_i,y_i): i=1,2,\ldots,k\}\subset \mathbb{C}^2.\end{equation*}
The problem of finding a polynomial $p \in \Pi_n$ which satisfies
the conditions
\begin{equation}\label{eq:intpr}
  p(x_i,y_i) = c_i, \quad i=1, \ldots, k,
\end{equation}
is called \emph{interpolation problem}. We denote this problem by $(\Pi_n, \mathcal X).$ The polynomial $p$ is called \emph{interpolating polynomial}.

\begin{definition}
The set of points ${\mathcal X}_k$ is called  \emph{$n$-poised}, if for any data $(c_1,\ldots, c_k)$,
there is a \emph{unique} polynomial $p\in\Pi_n$ satisfying the
conditions \eqref{eq:intpr}.
\end{definition}
\noindent By a Linear Algebra argument a necessary condition for
$n$-poisedness is
\begin{equation*}
  k = \#{\mathcal X}_k = \dim \Pi_n = N .
\end{equation*}

\begin{definition}
 The interpolating problem $(\Pi_n, {\mathcal X}_k)$ is called $n$-solvable, if for any data $(c_1, \ldots c_k)$, there exists a  (not  necessarily  unique) polynomial $p \in \Pi_n$ satisfying the conditions \eqref{eq:intpr}.
\end{definition}

% \noindent Also it is a fact in Linear Algebra that for any $N-1=(1/2)k(k+3)$ points in the plane there is a curve of degree $k$ passing through them.

 A polynomial $p\in\Pi_n$ is called
\emph{$n$-fundamental polynomial} of a point
$A\in{\mathcal X},$  if
\begin{equation*}\label{eq:fundpol}
  p(A) = 1 \qquad\text{and}\qquad  p\big\vert_{{\mathcal X}\setminus\{A\}} = 0,
\end{equation*}
where $p\big\vert_{\mathcal X}$ means the restriction of $p$ to
${\mathcal X}.$
We shall denote such a polynomial by $p_{A,{\mathcal X}}^\star.$

\noindent Sometimes we call $n$-fundamental also a polynomial from $\Pi_n$
that just vanishes at all the points of ${\mathcal X}$ but $A,$ since such
a polynomial is a nonzero constant multiple of $p_A^\star.$
A fundamental polynomial can be described as a plane curve containing all but one point of $\Xset.$

Next we consider an important concept of $n$-independence and $n$-dependence of point sets (see \cite{EGH96}, \cite{HJZ09a}, \cite{HM13}).

\begin{definition}
A set of points ${\mathcal X}$ is called \emph{$n$-independent}, if each its point has an $n$-fundamental polynomial. Otherwise, it is called \emph{$n$-dependent}.
\end{definition}
Since the fundamental polynomials are linearly independent, we get that $\#\mathcal{X} \le N$ is a necessary condition for n-independence.

\begin{proposition}
A set $\mathcal X$ is $n$-independent if and only if the interpolation problem $(\Pi_n, \mathcal X)$ is $n$-solvable.
\end{proposition}
\begin{proof}
Suppose $\mathcal X:={\mathcal X}_k.$ In the case of $n$-independence we have the following Lagrange formula for a polynomial $p\in\Pi_n$ satisfying interpolating conditions \eqref{eq:intpr}:
$$p = \sum_{i=1}^{k} c_i p_i^\star.$$

On the other hand if the interpolation problem is $n$-solvable then for each point $(x_i,y_i),\ i=1,\ldots,k,$ there exists an $n$-fundamental polynomial. Indeed, it is the solution of the interpolation problem \eqref{eq:intpr}, where $c_i = 1,$ and $c_j = 0\ \forall j \ne i.$
\end{proof}

\begin{definition}\label{def:depset}
 A set of points ${\mathcal X}$ is called
\emph{essentially $n$-dependent}, if none of its points has an $n$-fundamental polynomial.
\end{definition}

If a point set ${\mathcal X}$ is $n$-dependent, then for some $A\in{\mathcal X}$, there is
 no $n$-fundamental polynomial, which means
that for any polynomial $p\in\Pi_n$ we have that
\begin{equation*}
  p\big\vert_{{\mathcal X}\setminus\{A\}} = 0
  \quad \implies \quad p(A)=0.
\end{equation*}

Thus a set $\mathcal X$ is essentially $n$-dependent means that any
plane curve of degree $n$ containing all but one point of $\mathcal X$, contains
all of $\mathcal X$.

In the proof of the main result we will need the following
\begin{proposition}[\cite{HM15}, Cor.~2.2]
Suppose a set $\mathcal{X}$ is given. Denote by $\mathcal{Y}$ the subset of $\mathcal{X}$ that have $n$-fundamental polynomials with respect to $\mathcal{X}$. Then the set $\mathcal{X} \setminus \mathcal{Y}$ is essentially $n$-dependent.
\end{proposition}

\begin{corollary}\label{essential}
Any $n$-dependent point set has essentially $n$-dependent subset.
\end{corollary}

% \noindent In the case of independence we have the following Lagrange formula for a polynomial satisfying interpolating conditions \eqref{eq:intpr}.

% $$p = \sum_{i=1}^{k} c_i p_i^\star$$

% \noindent When $k=N,$ i.e., for a point set ${\mathcal X}_N,$
% the $n$-independence means $n$-poisedness.

\vspace{2mm}

Set
$$d(n,k):=\dim\Pi_n-\dim\Pi_{n-k}.$$
It is easily seen that $d(n,k)=(n+1)+n+\cdots+(n-k+2)=\frac{1}{2}k(2n-k+3),$ if $k \le n$.

In the sequel we will need the following well-known proposition (see, e.g., \cite{Raf11}, Proposition 3.1).

\begin{proposition}\label{prop:d(n,k)}
    Let $q$ be a curve of degree $k$ without multiple components and $k \le n$. Then the following assertions hold:

    \begin{enumerate}
        \item Any set of more than $d(n, k)$ points located on the curve $q$ is $n$-dependent;
        \item Any set $\mathcal{X}$ of $d(n, k)$ points located on the curve $q$ is $n$-independent if and only if
        $$p \in \Pi_n,\ p\big\vert_{\mathcal X}=0 \Rightarrow p=fq, \ \hbox{where}\ f \in \Pi_{n-k}.$$
    \end{enumerate}
\end{proposition}

\begin{corollary}\label{cor} The following assertions hold:
 \begin{enumerate}
\item  Any set of at least $n+2$ points located on a line is $n$-dependent;
\item  Any set of at least  $2n+2$ points located on a conic is $n$-dependent;
\item  Any set of at least  $3n+1$ points located on a cubic is $n$-dependent.
\end{enumerate}
\end{corollary}

\section{Some known results}
\vspace{2mm}
Let us start with the three known results which coincide with the first three items of the scale established in this paper, respectively.

\begin{theorem}[\cite{S}]\label{col:n+1}
Any set $\mathcal{X}$ consisting of at most $n+1$ points is $n$-independent.
\end{theorem}

\begin{theorem}[\cite{EGH96}, Prop.~1]\label{prop:2n+1}
A set $\mathcal{X}$ with no more than $2n + 2$ points on the plane is $n$-dependent if and only if either $n + 2$ of them are collinear or $\#\mathcal{X}=2n+2$ and all the $2n + 2$ points belong to a conic.
\end{theorem}

\begin{theorem}[\cite{HM12}, Thm.~5.1]\label{thm:3n}
A set $\Xset$ consisting of at most $3n$ points is $n$-dependent if
and only if at least one of the following conditions hold:

\vspace{-2mm}
\begin{enumerate}
    \setlength{\itemsep}{0mm}

    \item $n+2$ points are collinear;
    \item $2n+2$ points belong to a (possibly reducible) conic;
    \item $\#\mathcal{X} = 3n$, and there exist $\sigma_3 \in \Pi_3$ and $\sigma_n \in \Pi_n$ such that $\mathcal{X} = \sigma_3 \cap \sigma_n.$
\end{enumerate}
\end{theorem}

The following two results describe some properties of essentially dependent point sets laying in a curves of certain degrees.

\begin{proposition}[\cite{HV19}, Prop.~3.3]\label{ess-dep-on-curve-1}
Suppose that $m \le n$. If a set $\mathcal X$ of at most $mn$ points is essentially $\kappa$-dependent then all the points of $\mathcal X$ lay in a curve of degree $m$.
\end{proposition}

We say that a curve $\sigma$ is not empty with respect to a set $\mathcal X$ if $X \cap \sigma \neq \varnothing$.

\begin{theorem}[\cite{HV19}, Thm.~3.4]\label{ess-dep-on-curve-2}
Assume that $\sigma_m$ is a curve of degree $m$, which is either irreducible or is reducible such that all its irreducible components are not empty with respect to a set $\mathcal X\subset\sigma_m,$ where $\mathcal X$ is essentially $\kappa$-dependent and $m \le n + 2$. Then we have that $\#\mathcal X \ge mn$.
\end{theorem}

The next result states a necessary and sufficient conditions for a set of $mn$ points to coincide with the set of the intersection points of some two plane algebraic curves of degrees $m$ and $n$, respectively.

\begin{theorem}[\cite{HV19}, Thm.~3.1]\label{hv-main}
A set $\mathcal X$ with $\#\mathcal{X}=mn,\ m \le n$, is the set of intersection points of some two plane curves of degrees $m$ and $n$, respectively, if and only if the following two conditions are satisfied:

\vspace{-2mm}
\begin{enumerate}
    \setlength{\itemsep}{0mm}

    \item The set $\mathcal{X}$ is essentially $(m+n-3)$-dependent;
    \item No curve of degree less than $m$ contains all of $\mathcal X$.
\end{enumerate}
\end{theorem}

\section{Main result}

By combining Proposition \ref{ess-dep-on-curve-1} and Theorem \ref{ess-dep-on-curve-2} we readily get the following

\begin{proposition}\label{belongsCurve}
Suppose that $\mathcal{X}$ is an essentially $\kappa$-dependent point set with $\#\mathcal X\le mn-1$, where  $m \le n$, and $\kappa = m + n - 3$. Then there exists a number $r$, $1 \le r \le m-1$, and a curve $\sigma_r$ of degree $r$, such that the following conditions hold:
\begin{enumerate}
    \setlength{\itemsep}{0mm}
    \item $\#\mathcal{X} \ge rs$, where $r + s - 3 = \kappa$;
    \item $\sigma_r$ contains all of $\mathcal{X}$
    \item There is no curve of degree less than $r$ containing all of $\mathcal{X}$.
    
\end{enumerate}
\end{proposition}

\begin{proof}
We obtain from Proposition \ref{ess-dep-on-curve-1}  that the set of points $\mathcal{X}$ lies in a curve $\sigma$ of degree at most $m$. Without loss of generality we may assume that $\sigma$ is either irreducible or is reducible such that all its irreducible components are not empty with respect to the set $\mathcal X.$ Then, notice that the degree of the curve does not equal $m$, since in that case, in view of Theorem \ref{ess-dep-on-curve-2}, we would have that $\#X \ge mn$. Finally, consider such a curve $\sigma_r$ of the smallest possible degree  $1 \le r \le m - 1$. Note that $r < \kappa - r + 3$ since $r < m \le \kappa - m + 3$. Now, Theorem \ref{ess-dep-on-curve-2} implies that $\#\mathcal{X} \ge rs$, where $r + s - 3 = \kappa$.
\end{proof}

Now we are in a position to formulate the main result of the paper:

\begin{theorem}\label{main}
Suppose that $\mathcal{X}$ is a set of points such that $\#\mathcal X \le mn - 1$, where $m \le n$. Then $\mathcal{X}$  is $\kappa$-dependent, where $\kappa = m + n - 3$, if and only if there exist a number $r,\ 1\le r \le m - 1$, and an essentially  $\kappa$-dependent subset $\mathcal{Y} \subset \mathcal{X},\ $ $\#\mathcal{Y}\ge rs$, where $r + s - 3 = \kappa$, belonging to a curve of degree $r$, and not belonging to any curve of degree less than $r$. \\  Moreover, if $\#\mathcal{Y} = rs$ then we have that $\mathcal{Y}$ coincides with the set of intersection points of some two plane curves of degrees $r$ and $s$ respectively.
\end{theorem}
\begin{proof}
The sufficiency part is obvious. If some set has a $\kappa$-dependent subset then the set itself is $\kappa$-dependent. Now let us prove the part of necessity.

We have that the set $\mathcal{X}$ is $\kappa$-dependent. By the Corollary \ref{essential} there exists some essentially $\kappa$-dependent subset $\mathcal{Y} \subset \mathcal{X}$. 

Now, applying Proposition \ref{belongsCurve}
to the set of points $\mathcal Y$ we get that 
there exists a curve $\sigma_r$ of degree $r$, $1 \le r \le m-1$, containing all of $\mathcal{Y}$, such that $\#\mathcal{Y} \ge rs$, where $r + s - 3 = \kappa$,
and
there is no curve of lower degree containing all of $\mathcal{Y}.$

Finally, let us prove the "moreover" part of the theorem. Here we have that $\mathcal{Y}$ is essentially $\kappa$-dependent, $\#\mathcal{Y} = rs$ and there is no curve of degree less than $r$ passing through all the points of $\mathcal{Y}.$ Notice also that $r < s$ since $r < m < \kappa - r + 3.$ Hence we get from Theorem \ref{hv-main} that $\mathcal{Y}$ is the set of intersection points of some two plane curves of degrees $r$ and $s$, respectively. 
\end{proof}

Now, let us mention some necessary conditions for the set $\mathcal{X}$ to be able to apply Theorem \ref{main}. Suppose that we have a $\kappa$-dependent set $\mathcal{X}$. We need to find such numbers $m$ and $n$, for which $\#\mathcal{X} \le mn$, where $m + n - 3 = \kappa$ and $m \le n$. Since the the expression $mn$ achieves its maximum when $m = n = (\kappa + 3) / 2$, then $\#\mathcal{X}$ must be not more than $\floor{(\kappa+3)^2/4}$.

\section{Some special cases of Theorem \ref{main}}

In this section we verify that Theorem \ref{main} is a generalization of Theorems \ref{col:n+1}, \ref{prop:2n+1} and \ref{thm:3n}. For this purpose let us formulate Theorem \ref{main} in the special cases with $m=1, 2, 3, 4$.

\vspace{0.2cm}
\noindent{Case $m=1.$} \\
A set $\mathcal{X}$ of at most $\kappa + 1$ points is never $\kappa$-dependent.

\noindent This is equivalent to Theorem \ref{col:n+1}.

\vspace{0.2cm}
\noindent{Case $m=2.$} \\
A set $\mathcal{X}$ of at most $2\kappa + 1$ points is $\kappa$-dependent if and only if $\kappa + 2$ points of $\mathcal{X}$ are collinear.

\vspace{0.2cm}
\noindent{Case $m=3.$} \\
A set $\mathcal{X}$ of at most $3\kappa - 1$ points is $\kappa$-dependent, if and only if one of the following conditions hold:
\begin{enumerate}
    \setlength{\itemsep}{0mm}
    \item $\kappa + 2$ points of $\mathcal{X}$ belong to a line,
    \item $2\kappa + 2$ points of $\mathcal{X}$ belong to a conic.
\end{enumerate}

\noindent Clearly the statement in Case $m=3$ is a generalization of Theorem \ref{prop:2n+1}.

\vspace{0.2cm}
\noindent{Case $m=4.$} \\
A set $\mathcal{X}$ of at most $4\kappa - 5$ points is $\kappa$-dependent, if and only if one of the following conditions hold:
\begin{enumerate}
    \setlength{\itemsep}{0mm}
    \item $\kappa + 2$ points of $\mathcal{X}$ belong to a line,
    \item $2\kappa + 2$ points of $\mathcal{X}$ belong to a conic,
    \item $\#\mathcal{X} = 3\kappa$ and $\mathcal{X}$ coincides with an intersection points of some two algebraic curves of degrees $3$ and $\kappa$,
    \item more than $3\kappa$ points of $\mathcal{X}$ belong to a cubic.
\end{enumerate}

\noindent Finally, this statement generalizes Theorem \ref{thm:3n}.

Note that in above statements we used Corollary \ref{cor}.

\end{document}